\documentclass[12pt]{amsart}
\usepackage{graphicx,verbatim,amssymb}
\usepackage[active]{srcltx}

\newtheorem{thm}{Theorem}[section]
\newtheorem{cor}[thm]{Corollary}
\newtheorem{lem}[thm]{Lemma}
\newtheorem{prop}[thm]{Proposition}

\newtheorem*{thmA}{Theorem A}
\newtheorem*{thmB}{Theorem B}
\newtheorem*{thmC}{Theorem C}

\theoremstyle{definition}

\theoremstyle{remark}
\newtheorem*{rem}{Remark}
\numberwithin{equation}{section}

\newcommand{\R}{\mathbb R}
\def\C{\mathbb{C}}

\def\Rc{\mathcal{R}}

\def\Vc{\mathcal{V}}

\def\Z{\mathbb{Z}}

\def\0{\varnothing}
\def\1{\mathbb{I}}
\def\sm{\setminus}
\def\ol{\overline}

\def\Vol{\mathrm{Vol}}

\def\Int{\mathrm{Int}}
\def\vrt{\mathrm{Vert}}
\newcommand\pd[2]{\frac{\partial #1}{\partial #2}}
\renewcommand\le{\leqslant}
\renewcommand\ge{\geqslant}

\begin{document}

\title[On the theory of coconvex bodies]
{On the theory of coconvex bodies}
\author{Askold Khovanski\u{\i}}
\author{Vladlen Timorin}

\address[Askold Khovanskii]
{Department of Mathematics, University of Toronto, Toronto,
Canada; Moscow Independent University; Institute for Systems Analysis,
Russian Academy of Sciences.
}

\email{askold@math.toronto.edu}

\address[Vladlen Timorin]{Faculty of Mathematics and Laboratory of Algebraic Geometry,
National Research University Higher School of Economics,
7 Vavilova St 117312 Moscow, Russia;
Independent University of Moscow,
Bolshoy Vlasyevskiy Pereulok 11, 119002 Moscow, Russia}

\email{vtimorin@hse.ru}

\thanks{
The first named author was partially supported by Canadian Grant \textrm{N~156833-02}.
The second named author was partially supported by
the Dynasty Foundation grant, RFBR grants 11-01-00654-a, 12-01-33020, 13-01-12449,
and AG Laboratory NRU HSE, MESRF grant ag. 11 11.G34.31.0023
}


\begin{abstract}
If the complement of a closed convex set in a closed convex cone
is bounded, then this complement minus the apex of the cone is called
a coconvex set.
Coconvex sets appear in singularity theory (they are
closely related to Newton diagrams) and in commutative algebra.
Such invariants of coconvex sets as volumes, mixed volumes, number
of integer points, etc., play an important role.
This paper aims at extending various results from the theory of convex
bodies to the coconvex setting.
These include the Aleksandrov--Fenchel inequality and the Ehrhart duality.
\end{abstract}
\maketitle

\section{Introduction}

Geometric study of coconvex bodies is motivated by singularity theory.
The connections between coconvex geometry and singularity theory are similar
to the connections between convex geometry and algebraic geometry.
Many local phenomena studied by singularity theory are local manifestations
of global algebraic geometry phenomena.
Thus it would be natural to expect that many properties of coconvex bodies
are manifestations of properties of convex bodies.
In this paper, we prove a number of results of this spirit.

In the first subsection of the introduction, we briefly overview the
connections of convex geometry with algebraic geometry, of algebraic
geometry with singularity theory and, finally, of singularity theory with
coconvex geometry. 
These were the main motivations of the authors, however, neither
algebraic geometry, nor singularity theory appear later in the text.
Thus the following subsection can be omitted.

\subsection{Overview}
The theory of Newton polytopes founded in 1970s revealed unexpected connections between
algebraic geometry and convex geometry.
These connections turned out to be useful for both fields.
According to a theorem of Kouchnirenko and Bernstein \cite{Ko,B}, the number of solutions
of a polynomial system $P_1=\cdots=P_n=0$ in $(\C\sm\{0\})^n$ equals the mixed
volume of the corresponding Newton polytopes $\Delta_1$, $\dots$, $\Delta_n$.
The relationship between Algebra and Geometry contained in this result allowed to
prove the Aleksandrov--Fenchel inequalities using transparent
and intuitive algebraic geometry considerations \cite{Kh,T}, to find
previously unknown analogs of the Aleksandrov--Fenchel inequalities in algebraic geometry
\cite{KKa}, to find convex-geometric versions of the
Hodge--Riemann relations that generalize the polytopal Aleksandrov--Fenchel
inequalities \cite{Mc,VT}.
The number of integer points in polytopes is a classical object of study in
geometry and combinatorics.
The relationship with algebraic geometry has enriched this area with explicit
formulas of Riemann-Roch type \cite{KPb}, has allowed to connect the Ehrhart duality
with the Serre duality (from topology of algebraic varieties).
The integration with respect to the Euler characteristic, inspired by the
connections with algebraic geometry, has allowed to find a better viewpoint
on classical results of P. McMullen, simplify and considerably generalize them
\cite{KPa,KPb}.
Connections with algebraic geometry led to important results in combinatorics
of simple (and non-simple) convex polytopes, and had many more follow-ups.

The theory of Newton polytopes has a local version, which studies singularities
of sufficiently generic polynomials with given \emph{Newton diagrams} at the origin.
This theory connects singularity theory with somewhat unusual geometric objects,
namely, Newton diagrams.
A Newton diagram is the union of all compact faces of an unbounded convex polyhedron
lying in a convex cone (which in this case coincides with the positive coordinate orthant)
and coinciding with the cone sufficiently far from the origin.
The complement in the cone of the given unbounded convex polyhedron is, in
our terminology, a \emph{coconvex body} (except that it is also convenient
to remove the apex of the cone from the coconvex body for reasons that will
become clear later).
Computations of local invariants in algebraic geometry and singularity theory
have persistently led to volumes (and mixed volumes) of coconvex bodies, the number of integer points
in coconvex bodies, etc.
Computation of local invariants is often reduced to computation of global invariants.
Let us illustrate this effect on the following toy problem: compute the multiplicity of
the zero root of a polynomial $P(z)=a_kz^k+\cdots+a_nz^n$ with $a_k\ne 0$ and $a_n\ne 0$.
The multiplicity of the zero root of the polynomial $P$ equals the number of
nonzero roots of a polynomial $P_\varepsilon=P+\varepsilon$ that vanish (i.e., tend to $0$) as $\varepsilon\to 0$.
The Newton polytopes of the polynomials $P_\varepsilon$ ($\varepsilon\ne 0$) and $P$
are, respectively, the intervals $[0,n]$ and $[k,n]$.
The lengths $n$ and $n-k$ of these intervals are equal to the number of nonzero roots of
the polynomials $P_\varepsilon$ and $P$ (the answers in global problems).
It follows that exactly $k=n-(n-k)$ solutions vanish as $\varepsilon\to 0$.
Thus, in this simplest case, the multiplicity $\mu=k$ of the root 0 (the local invariant)
equals the difference of two global invariants, namely, the lengths $n$ and $n-k$ of the
intervals $[0,n]$ and $[k,n]$.

Similarly to this simple example, many questions of singularity theory (local questions)
reduce to questions of algebraic geometry (global questions).
Computing various local invariants for generic collections of functions with
given Newton diagrams reduces to computing global algebro-geometric invariants
for collections of generic polynomials with given Newton polytopes.

A systematic development of the coconvex bodies theory became a pressing need when,
several years ago, relationships between convex and coconvex geometry on one side,
algebraic geometry and singularity theory on the other side, were found that are far
more general than relationships based on Newton diagrams and Newton polytopes.
For example, these relationships have allowed Kaveh and Khovanskii
\cite{KK} to deduce non-trivial commutative algebra
inequalities from a version of the Brunn--Minkowski inequality for coconvex bodies
(this version follows from Theorem A).
Computing Hilbert polynomials of algebraic varieties and their local versions for
algebraic singularities leads to problems of counting integer points in lattice
convex and coconvex polytopes.

\subsection{Terminology and notation}
We start by recalling some terminology from convex geometry.
The \emph{Minkowski sum} of two convex sets $A$, $B\subset \R^d$ is defined as
$A+B=\{a+b\,|\, a\in A,\ b\in B\}$.
For a positive real number $\lambda$, we let $\lambda A$ denote the
set $\{\lambda a\,|\, a\in A\}$.
By definition, a \emph{convex body} is a compact convex set, whose interior
is nonempty.

Let $C\subset\R^d$ be a closed strictly convex cone with the apex at $0$
and a nonempty interior.
Consider a closed convex subset $\Delta\subset C$ such that $C\sm\Delta$
is bounded and nonempty.
Then the set $A=C\sm(\Delta\cup\{0\})$ is called a \emph{coconvex body}.
When we talk about volumes, we may replace $A$ with its closure $\ol A$.
However, for the discussion of integer points in coconvex bodies, the
distinction between $A$ and $\ol A$ becomes important.
If $A$ and $B$ are coconvex sets with respect to the same cone $C$,
then we can define $A\oplus B$ as $C\sm((\Delta_A+\Delta_B)\cup\{0\})$,
where $\Delta_A$ and $\Delta_B$ are the unbounded components of $C\sm A$
and $C\sm B$, repsectively.
It is clear that any $C$-coconvex set can be represented as
a set-theoretic difference of two convex bounded sets.
This representation allows to carry over a number of results concerning
the convex bodies with the operation $+$ to the coconvex
bodies with the operation $\oplus$.

In this paper, we describe several results of this type.
Although the reduction from the ``convex world'' to the ``coconvex world'' is always
simple and sometimes straightforward, the results obtained with the
help of it are interesting because, firstly, they are related
(through Newton diagrams) with singularity theory \cite{AVG,KV} and
commutative algebra \cite{KK}, and, secondly, they are intrinsic, i.e., do not
depend on a particular representation of a coconvex set as a difference of two convex sets.
Coconvex Aleksandrov--Fenchel inequalities also appeared in \cite{Fi}
in the context of Fuchsian groups.
This paper extends the earlier very short preprint \cite{KT} of the authors,
in which just the coconvex Aleksandrov--Fenchel inequality has been discussed.

\subsection{Aleksandrov--Fenchel inequalities}
A \emph{linear family of convex bodies} is a collection of the following
objects: a real vector space $\Vc$, an open subset $\Omega\subset\Vc$,
a map $f$ from $\Omega$ to the set of all convex bodies in $\R^d$ such that
$$
f(\lambda_1 v_1+\dots+\lambda_n v_n)=\lambda_1 f(v_1)+\dots +\lambda_n f(v_n)
$$
whenever all $v_i\in\Omega$,
all $\lambda_i$ are positive, and $\lambda_1 v_1+\cdots+ \lambda_n v_n\in\Omega$.
A linear family of convex bodies \emph{with $m$ marked points} is a linear family $(\Vc,\Omega,f)$ of convex bodies, in which some $m$ elements of $\Omega$ are marked.

With every linear family $\alpha=(\Vc,\Omega,f)$ of convex bodies, we associate
the \emph{volume polynomial} $\Vol_\alpha$ on $\Vc$ as follows.
For $v\in\Omega$, we define $\Vol_\alpha(v)$ as the usual $d$-dimensional
volume of the convex body $f(v)$.
It is well known that the function $\Vol_\alpha$ thus defined extends
to a unique polynomial on $\Vc$ that is homogeneous of degree $d$.
For $v\in\Vc$, we let $L_v$ denote the usual directional (Lie) derivative along $v$.
Thus $L_v$ is a differential operator that acts on functions, in particular,
degree $k$ polynomials on $\Vc$ are mapped by this operator to degree $k-1$
polynomials.
If $\alpha=(\Vc,\Omega,f)$ is a linear family
of convex bodies with $(d-2)$ marked points $v_1$, $\dots$, $v_{d-2}\in\Omega$,
then we define the \emph{Aleksandrov--Fenchel} symmetric bilinear form $B_\alpha$ on $\Vc$
by the formula
$$
B_\alpha(u_1,u_2)=\frac 1{d!}L_{u_1}L_{u_2}L_{v_1}\dots L_{v_{d-2}}(\Vol_\alpha).
$$
Note that the expression in the right-hand side is a real number
(called the \emph{mixed volume} of the convex bodies $f(u_1)$, $f(u_2)$, $f(v_1)$, $\dots$,
$f(v_{d-2})$ provided that $u_1$, $u_2\in\Omega$).
Indeed, this is the result of the action of a homogeneous degree $d$ differential operator
with constant coefficients on a homogeneous degree $d$ polynomial.
The corresponding quadratic form $Q_\alpha(u)=B_\alpha(u,u)$ is given by
the formula $Q_\alpha=\frac 2{d!}L_{v_1}\dots L_{v_{d-2}}(\Vol_\alpha)$.
The expression in the right-hand side is the result of the action of
a degree $d-2$ homogeneous differential operator with constant coefficients on a homogeneous
degree $d$ polynomial, i.e., a quadratic form.
\emph{The Aleksandrov--Fenchel inequality} \cite{Al} states that, for all $u_1\in\Vc$ and
$u_2\in\Omega$, we have
$$
B_\alpha(u_1,u_2)^2\ge B_\alpha(u_1,u_1) B_\alpha(u_2,u_2).
$$
The Aleksandrov--Fenchel inequality is a far-reaching generalization of
the classical isoperimetric inequality. See \cite{Mc,VT} for generalizations
of the Aleksandrov--Fenchel inequality for convex polytopes.

The following are standard corollaries of the Aleksandrov--Fenchel inequality:

\noindent\emph{Brunn--Minkowski inequality:} the function $\Vol_\alpha^{\frac 1d}$ is
concave, i.e.
$$
(\Vol_\alpha(tu+(1-t)v))^\frac 1d\ge t\Vol_\alpha(u)^\frac 1d+(1-t)\Vol_\alpha(v)^\frac 1d,\quad t\in [0,1].
$$

\noindent\emph{Generalized Brunn--Minkowski inequality:} the function
$\left(L_{v_1}\dots L_{v_k}\Vol_\alpha\right)^\frac 1{d-k}$ is concave.

\noindent\emph{First Minkowski inequality:}
$$
\left(\frac 1{d!}L_{u}L_{v}^{d-1}(\Vol_\alpha)\right)^d\ge \Vol_\alpha(u)\Vol_\alpha(v)^{d-1}.
$$

\noindent\emph{Second Minkowski inequality:} if all marked points coincide with $u$, then
$$
B_\alpha(u,v)^2\ge \Vol_\alpha(u)\, B_\alpha(v,v)
$$

\subsection{Coconvex Aleksandrov--Fenchel inequalities}

Define a \emph{linear family of $C$-coconvex bodies} as a collection of
the following objects: a vector space $\Vc$,
an open subset $\Omega\subset\Vc$, a map $g$ from $\Omega$ to the set
of all $C$-coconvex bodies such that
$$
g(\lambda_1 v_1+\dots +\lambda_n v_n)=\lambda_1 g(v_1)\oplus\dots \oplus\lambda_n g(v_n)
$$
whenever all $v_i\in\Omega$, all $\lambda_i$ are positive, and
$\lambda_1 v_1+\dots +\lambda_n v_n\in\Omega$.
A linear family of $C$-coconvex bodies \emph{with $m$ marked points} is a
linear family $(\Vc,\Omega,g)$ of $C$-coconvex bodies,
in which some $m$ elements of $\Omega$ are marked.
With every linear family $\beta$ of $C$-coconvex bodies, we associate the
volume function $\Vol_\beta$ in the same way as with a linear family of convex bodies.
The function $\Vol_\beta$ thus defined is also a homogeneous degree $d$ polynomial
(we will prove this below).
Given a linear family $\beta$ of $C$-coconvex bodies with $d-2$ marked points
$v_1$, $\dots$, $v_{d-2}$, we define the \emph{coconvex Aleksandrov--Fenchel}
symmetric bilinear form as
$$
B_\beta^C(u_1,u_2)=\frac 1{d!}L_{u_1}L_{u_2}L_{v_1}\dots L_{v_{d-2}}(\Vol_\beta).
$$
We will also consider the corresponding quadratic form
$Q^C_\beta=\frac 2{d!}L_{v_1}\dots L_{v_{d-2}}(\Vol_\beta)$.
The following is one of our main results.

\begin{thmA}
  The form $Q_\beta^C$ is non-negative, i.e. $Q_\beta^C(u)\ge 0$
  for all $u\in\Vc$.
  In particular, the corresponding symmetric bilinear form
  satisfies the Cauchy--Schwarz inequality
  $$
  B_\beta^C(u_1,u_2)^2\le B_\beta^C(u_1,u_1) B_\beta^C(u_2,u_2).
  $$
\end{thmA}

The inequality stated in Theorem A is called the \emph{coconvex
Aleksandrov--Fenchel inequality}.
In recent paper \cite{Fi}, Theorem A is proved under the assumption
that $C$ is a fundamental cone of some Fuchsian group $\Gamma$ acting by linear isometries
of a pseudo-Euclidean metric, and $C\sm (g(v)\cup\{0\})$ is the intersection of some convex
$\Gamma$-invariant set with $C$, for every $v\in\Omega$.
Theorem A is motivated by an Aleksandrov--Fenchel type inequality
for (mixed) intersection multiplicities of ideals \cite{KK}.

The following inequalities can be derived from Theorem A in the same way as
similar inequalities for convex bodies follow from the classical Alexandrov--Fenchel
inequality (cf. \cite{Fi}):

\noindent\emph{Reversed Brunn--Minkowski inequality:} the function $\Vol_\beta^{\frac 1d}$ is convex, i.e.
$$
(\Vol_\beta(tu+(1-t)v))^\frac 1d\le t\Vol_\beta(u)^\frac 1d+(1-t)\Vol_\beta(v)^\frac 1d,\quad t\in [0,1].
$$

\noindent\emph{Generalized reversed Brunn--Minkowski inequality:} the function
$\left(L_{v_1}\dots L_{v_k}\Vol_\beta\right)^\frac 1{d-k}$ is convex.

\noindent\emph{First reversed Minkowski inequality:}
$$
\left(\frac 1{d!}L_{u}L_{v}^{d-1}(\Vol_\beta)\right)^d\le \Vol_\beta(u)\Vol_\beta(v)^{d-1}.
$$

\noindent\emph{Second reversed Minkowski inequality:} if all marked points coincide with $u$, then
$$
B_\beta^C(u,v)^2\le \Vol_\beta(u)\, B_\beta^C(v,v)
$$

\subsection{Coconvex polytopes as virtual convex polytopes}
The set of convex polytopes is closed under Minkowski addition
but not closed under ``Minkowski subtraction''.
Virtual convex polytopes are geometric objects introduced in \cite{KPa}
that can be identified with formal Minkowski differences of convex polytopes.
We will now briefly recall the notion of a virtual convex polytope.

Consider the smallest ring $\Rc(\R^d)$ of sets containing all closed
half-planes in $\R^d$.
Clearly, all convex polytopes belong to $\Rc(\R^d)$.
There is a unique finitely additive measure $\chi$ on $\Rc(\R^d)$ such that,
for every closed bounded set $A\in\Rc(\R^d)$, the number $\chi(A)$ is equal
to the Euler characteristic of $A$.
Let $Z(\R^d)$ be the $\Z$-algebra of all measurable functions with respect to
$\Rc(\R^d)$ with values in $\Z$, and let $Z_c(\R^d)$ be its subalgebra
consisting of functions with compact support.
With every element $A\in\Rc(\R^d)$, we associate its \emph{indicator function}
$\1_A$ that is equal to 1 on $A$ and to 0 elsewhere.
The additive group of $Z_c(\R^d)$ is spanned by the indicator functions of
convex polytopes.
If $\alpha$, $\beta\in Z_c(\R^d)$, we define the \emph{Minkowski product}
$\alpha*\beta$ as the convolution
$$
\alpha*\beta(x)=\int \alpha(y)\beta(x-y) d\chi(y).
$$
(The integration is performed over $\R^d$ with respect to the measure $\chi$.)
We recall one simple lemma from \cite{KPa}:

\begin{lem}
\label{l:mink-prod}
  Suppose that a convex polytope $C$ in $\R^d$ is equal to the Minkowski
  sum of convex polytopes $A$ and $B$.
  Then $\1_C=\1_A*\1_B$.
\end{lem}

Thus the Minkowski sum of convex polytopes corresponds to the Minkowski
product in $Z_c(\R^d)$.

\begin{proof}
  Consider the function
  $$
  \varphi(x)=\int \1_A(y)\1_B(x-y) d\chi(y).
  $$
  If $x\in C$, then the set $F(y)$ of all $y\in A$ such that $x-y\in B$
  is a nonempty convex polytope (clearly, both conditions
  $y\in A$ and $x-y\in B$ define systems of linear inequalities on $y$).
  Therefore, $\varphi(x)=1$.
  If $x\not\in C$, then the set $F(y)$ is empty; therefore, $\varphi(x)=0$.
  We see that $\varphi$ coincides with $\1_C$.
\end{proof}

It follows from Lemma \ref{l:mink-prod} that $\1_{\{0\}}$ is the identity
element of the ring $Z_c(\R^d)$.
It is proved in \cite{KPa} that, for every convex polytope $A$ in $\R^d$,
the indicator function $\1_A$ is an invertible element of the ring $Z_c(\R^d)$, i.e.,
there exists an element $\varphi\in Z_c(\R^d)$ with the property
$\varphi*\1_A=\1_{\{0\}}$ (we write $\varphi=\1_A^{-1}$).
The function $\varphi:\R^d\to\Z$ admits a simple explicit description:
it is equal to $(-1)^d$ at all interior points of the set
$\{-x\,|\, x\in A\}$ and to 0 at all other points.
\emph{Virtual $($convex$)$ polytopes} are defined as elements of $Z_c(\R^d)$
of the form $\1_A*\1_B^{-1}$, where $A$ and $B$ are convex polytopes.
If we identify convex polytopes with their indicator functions, then
virtual polytopes are identified with formal Minkowski differences
of convex polytopes.
Virtual polytopes form a commutative group under Minkowski multiplication.

The following theorem is a general principle that allows to reduce
various facts about coconvex polytopes to the corresponding facts
about convex polytopes.

\begin{thmB}
Let $C\subset\R^d$ be a closed strictly convex cone with the apex at $0$
and a nonempty interior.
Suppose that $\Delta\subset C$ is a convex closed subset with a bounded
complement, and $A=C\sm (\Delta\cup\{0\})$ is the corresponding
$C$-coconvex body.
Fix a linear function $\xi:\R^d\to\R$ such that $\xi\ge 0$ on $C$
and $\xi^{-1}(0)\cap C=\{0\}$; set $\Delta_t=\Delta\cap\{\xi\le t\}$.
\begin{enumerate}
\item
The function $-\1_A$ is a virtual polytope. Moreover, we have
$$
-\1_A=\1_{\Delta_t}*\1_{C_t}^{-1}.
$$
for all sufficiently large $t\in\R$.
\item
If $A$ and $B$ are $C$-coconvex bodies, then
$$
-\1_{A\oplus B}=(-\1_A)*(-\1_B).
$$
\end{enumerate}
\end{thmB}

Theorem B explains our definition of a coconvex body, in particular, the
choice of the boundary points that need to be included into it.
Theorem A can be deduced from Theorem B and the convex Aleksandrov--Fenchel inequalities.
There are many other consequences of Theorem B that deal with $C$-coconvex
integer polytopes.
Some of them are stated below.

Define a \emph{convex integer polytope} as a convex polytope with integer vertices.
Let $\Rc_c(\Z^d)$ be the minimal subring of the ring of sets $\Rc(\R^d)$ containing
all convex integer polytopes.
Similarly, let $Z_c(\Z^d)$ be the minimal subalgebra of $Z_c(\R^d)$ containing
the indicator functions of all convex integer polytopes.
Clearly, any function in $Z_c(\Z^d)$ is measurable with respect to $\Rc_c(\Z^d)$.
A \emph{valuation on integer polytopes} is by definition a finitely additive
measure on $\Rc_c(\Z^d)$.
For a valuation $\mu$ on integer polytopes and an element $\varphi$ of
the ring $Z_c(\R^d)$, we define $\mu(\varphi)$ as the integral of the function
$\varphi$ with respect to the measure $\mu$.

Recall that a function $P$ on a commutative multiplicative group $G$
is said to be \emph{polynomial} of degree $\le d$ if, for every fixed $g\in G$,
the function $P(gx)-P(x)$ is polynomial of degree $\le d-1$.
Polynomial functions of degree 0 are by definition constant functions.
Define the group of \emph{virtual integer polytopes} as the subgroup
of the group of virtual polytopes generated by the indicator functions of
all convex integer polytopes.
Recall the following theorem of \cite{KPa}:
\emph{if a valuation $\mu$ on integer polytopes is polynomial of degree $\le k$,
i.e., for every convex integer polytope $A$, the function $x\mapsto \mu(A+x)$
is a polynomial on $\Z^d$ of degree at most $k$, then the function
$\varphi\mapsto \mu(\varphi)$ is a polynomial function
of degree $\le d+k$ on the group of virtual integer polytopes with $*$
as the group operation}.
An important example of a valuation on integer polytopes is the valuation
$\mu$ that assigns the number of integer points in $X$ to every $X\in\Rc_c(\Z^d)$.
This valuation can be evaluated on all virtual integer polytopes.
In particular, for every integer convex polytope $A$, the number
$\mu(\1_A^{*n})$ depends polynomially on $n$ (cf. \cite{Mc77,KPa}).
This polynomial function is called the \emph{Ehrhart polynomial} of $A$.

Let $C$, $\Delta$ and $A$ be as in Theorem B.
Suppose that $C$ is a \emph{integer cone}, i.e., there exist elements
$v_1$, $\dots$, $v_n\in\Z^d$ such that
$$
C=\{t_1v_1+\cdots+t_nv_n\,|\, t_1,\dots,t_n\ge 0\}.
$$
Suppose also that all vertices of $\Delta$ belong to $\Z^d$
(in this case, we will say that $A$ is a \emph{$C$-coconvex integer polytope}).
The following statements are corollaries of Theorem B.

\begin{cor}
\label{c:poly}
  Let $\mu$ be a polynomial valuation of degree $\le k$ on integer convex polytopes,
  and let $A_1$, $\dots$, $A_n$ be $C$-coconvex integer polytopes.
  The number
  $$
  E(m_1,\dots,m_n)=\mu(m_1A_1\oplus m_2A_2\oplus\cdots\oplus m_nA_n)
  $$
  is defined for any choice of positive integers $m_1$, $\dots$, $m_n$,
  and depends polynomially $($of degree $\le k+d)$ on $m_1$, $\dots$, $m_n$.
\end{cor}

\begin{cor}
  \label{c:Ehrhart}
  The polynomial function $E$ thus obtained can be evaluated at
  $m_1=\cdots=m_n=-1$, and we have
  $$
  E(-1,\dots,-1)=(-1)^d\mu(X)-\mu(\{0\}),
  $$
  where the set $X\in\Rc_c(\Z^d)$ consists of all vectors $-v_1-\dots-v_n$, where
  each $v_i$ belongs to the closure of $A_i$ and to the interior of $C$.
\end{cor}

\subsection{Generating functions for integer points}
Let us now consider the following valuation $G$ on integer convex polytopes
with values in rational functions of some formal variables $x_1$, $\dots$, $x_d$:
$$
G(X)=\sum_{a\in X} x^a.
$$
Here $X$ is any element of $\Rc_c(\Z^d)$, and $x^a$ is the monomial
$x_1^{a_1}\cdots x_d^{a_d}$.
The rational function $G(X)$ is called the \emph{generating function for the integer
points in $X$}.
The following theorem is due to Brion \cite{Br}, see also \cite{KPb,BS}:
\emph{the valuation $G$ extends to the minimal ring of sets $\Rc(\Z^d)$ containing
all integer cones and their parallel translations by integer vectors; moreover,
for any integer convex polytope $A$, we have
$$
G(A)=\sum_{a\in\vrt(A)} x^a G(C_a),
$$
where $\vrt(A)$ is the set of all vertices of $A$,
and $C_a$ is the cone spanned by $A-a$ $($the difference is in the sense of Minkowski$)$.}

Note that $G(C_a)$ can be computed explicitly, by subdividing $C_a$ into
cones, each of which is spanned by a basis of $\Z^d$.
If a cone $C$ is spanned by a basis of $\Z^d$, then the computation of
$G(C)$ reduces to the summation of a geometric series.
In this paper, we will prove the following theorem that generalizes
Brion's theorem to coconvex polytopes.

\begin{thmC}
  Let $C$ be an integer cone, and let $A$ be a $C$-coconvex integer polytope.
  Then
  $$
  G(A)=-\sum_{a\in\vrt(A)\sm\{0\}} x^a G(C'_a)+G(C)-1,
  $$
  where $C'_a$ is the cone with the apex at $0$ such that a small
  neighborhood of $0$ in $C'_a$ coincides with a small neighborhood of
  $0$ in $\ol{C\setminus A}-a$.
\end{thmC}

Observe that usual exponential sums over the integer points of $A$ are obtained from the
rational function $G(A)$ by substituting the exponentials $e^{p_1}$, $\dots$, $e^{p_d}$ for
the variables $x_1$, $\dots$, $x_d$.
Sums of quasi-polynomials (in particular, sums of polynomials)
can be obtained from exponential sums by differentiation with respect to parameters
$p_1$, $\dots$, $p_d$.
A similar theory exists for integrals of exponentials, quasi-polynomials, etc.,
over convex or coconvex polytopes.

\begin{rem}
 In order to compute the number of integer points in $A$, one is tempted to
 substitute $x_i=1$ into the expression for the rational function $G(A)$
 through the generating functions of cones.
 However, this is problematic as the denominator of this expression vanishes
 at the point $(1,1,\dots,1)$.
 To obtain a numeric value, one can, e.g., choose a generic line passing through the
 point $(1,1,\dots,1)$, consider the Laurent series expansion of $G(A)$ along
 this line, and then take the free coefficient of this Laurent series.
 The same procedure is applicable to computing a quasi-polynomial sum over $A$
 for exceptional values of $p=(p_1,\dots,p_d)$, for which the rational function of
 $e^{p_1}$, $\dots$, $e^{p_d}$, equal to this quasi-polynomial sum at generic points, has a pole.
\end{rem}

\section{Proof of Theorem A}
Recall that every quadratic form $Q$ on a finite dimensional real vector space $\Vc$ can be represented in
the form
$$
x_1^2+\cdots+x_k^2-x_{k+1}^2-\cdots-x_{k+\ell}^2
$$
for a suitable linear coordinate system $(x_1,\dots,x_m)$, $m\ge k+\ell$.
The pair $(k,\ell)$ is called the \emph{signature} of $Q$.
It is well known that $Q$ has signature $(1,\ell)$ for some $\ell$ if and only if
both of the following conditions are fulfilled:
\begin{enumerate}
  \item there exists a vector $v_0\in\Vc$ with $Q(v_0)>0$;
  \item the corresponding symmetric bilinear form $B$ (such that $B(u,u)=Q(u)$)
  satisfies the reversed Cauchy--Schwarz inequality: $B(u,v)^2\ge Q(u)Q(v)$
  for all $u\in\Vc$ and $v\in\Vc$ such that $Q(v)>0$.
\end{enumerate}
Thus, the Aleksandrov--Fenchel inequality is equivalent to the fact that
$Q_\alpha$ has signature $(1,\ell(\alpha))$ for every finite-dimensional linear
family $\alpha$ of convex $d$-dimensional bodies with $d-2$ marked points.

\subsection{An illustration in the case $d=2$}
To illustrate Theorem A in a simple case, we set $d=2$.
Then $C$ is a wedge.
\begin{figure}
  \includegraphics[height=6cm]{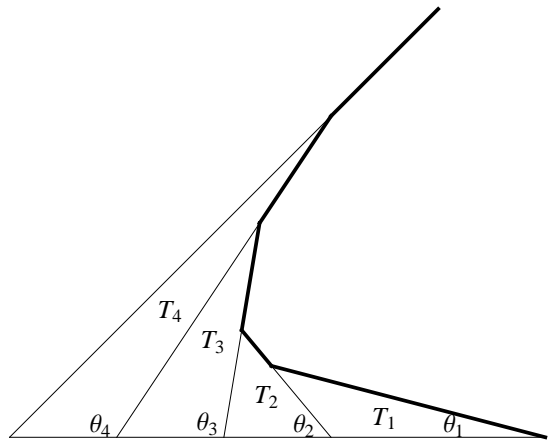}
  \caption{Co-convex body in the plane}
  \label{f:plane}
\end{figure}
Introduce an affine coordinate system $(x,y)$ in $\R^2$.
We may assume that $C$ lies in the upper half-plane $y\ge 0$, and that
one of the two boundary rays of $C$ coincides with the positive $x$-semiaxis.
Suppose that a $C$-coconvex body $A$ is a polygon with edges $E_1(A)$, $\dots$, $E_n(A)$,
ordered from right to left.
Let $\theta_k$ denote the angle between the line containing the edge $E_k(A)$ and the
$x$-axis, as indicated in Figure \ref{f:plane}.
Here and in the sequel, $k$ runs through 1, $\dots$, $n$.
Fix a Euclidean metric on $\R^2$.
We will write $h_k(A)$ for the distance from the origin to the line containing
the edge $E_k(A)$.

If we now have two co-convex polygons $A$ and $B$ with the same number of edges
and with the same angles $\theta_k$, then we have
$$
h_k(A\oplus B)=h_k(A)+h_k(B).
$$
Thus, if the angles $\theta_k$ are fixed, co-convex polygons can be identified with
vectors $(h_1(A),\dots,h_n(A))$, and the operation $\oplus$ on $C$-coconvex polygons
corresponds to the usual vector addition.
The numbers $h_k(A)$ are called \emph{support numbers} of $A$.
We can now consider the following linear family $\beta=(\Vc,\Omega,g)$ of
$C$-coconvex polygons.
The space $\Vc$ is $\R^n$ (the dimension of the space being equal to
the number of support numbers), $\Omega$ is a small neighborhood of
$(h_1(A),\dots,h_n(A))$, and $g(h_1,\dots,h_n)$ is the $C$-coconvex
polygon with support numbers $h_1$, $\dots$, $h_n$.

A $C$-coconvex polygon $g(v)$ can be naturally represented as a union of triangles
$T_1(v)$, $\dots$, $T_n(v)$, as shown in Figure \ref{f:plane}.
The angles of these triangles depend only on $\theta_1$, $\dots$, $\theta_n$.
Namely, $T_k(v)$ has angles $\theta_k$, $\pi-\theta_{k+1}$, $\theta_{k+1}-\theta_k$.
It follows that $T_k(v)$ and $T_k(w)$ are similar.
We conclude that the area of $T_k(v)$ is the square of some linear functional
$\varphi_k$ on $\Vc$.
Indeed, if the angles of a triangle are fixed, then its area is equal,
up to a constant positive factor depending only on the angles, to the square of the length
of any chosen edge.
The length $\lambda_k$ of the horizontal edge of $T_k(v)$ can be expressed
as a linear functional on $\Vc$:
$$
\lambda_k=\frac{h_k}{\sin\theta_k}-\frac{h_{k+1}}{\sin\theta_{k+1}}
$$
(if $k=n$, then the second term in the right-hand side should be omitted).

We can now write the following formula for the area of $f(v)$:
$$
\Vol_\beta(v)=\varphi_1(v)^2+\dots+\varphi_n(v)^2.
$$
It is obvious from this representation that $\Vol_\beta(v)$ is a positive definite
quadratic form on $\Vc$.
The Cauchy--Schwarz inequality written for this quadratic form yields
the inequality
$$
B^C_\beta(u,v)^2\le \Vol_\beta(u)\Vol_\beta(v).
$$
This is the reversed Aleksandrov--Fenchel inequality for $d=2$
(for a special choice of a linear family $\beta$, but in fact the
general case easily reduces to this special case).
The argument presented above imitates some classical proofs of the isoperimetric
inequality.

\subsection{Reduction of Theorem A to the Aleksandrov--Fenchel inequality}
\label{ss:AF}
In this subsection, we prove Theorem A.
The proof is a reduction to the convex Aleksandrov--Fenchel inequality.
Fix a strictly convex closed cone $C$ with the apex at $0$ and a linear family
$\beta=(\Vc,\Omega,g)$ of $C$-coconvex bodies.
We may assume that $g(\Omega)$ is bounded in the sense that there is
a large ball in $\R^d$ that contains all coconvex bodies $g(v)$, $v\in\Omega$.
Consider a linear functional $\xi$ on $\R^d$ such that $\xi\ge 0$ on $C$ and
$\xi^{-1}(0)\cap C=\{0\}$.
For every subset $A\subset C$, we let $A_t$ denote the set of all
points $a\in A$ such that $\xi(a)\le t$.
Since $g(\Omega)$ is bounded, there exists a real number $t_0>0$
such that $g(v)=g(v)_{t_0}$ for all $v\in\Omega$.

Choose any $t_1>t_0$.
We will now define a linear family
$\alpha=(\Vc\times\R,\Omega\times (t_0,t_1),f)$ of convex bodies as follows.
For $v\in\Omega$ and $t\in (t_0,t_1)$, we set $f(v,t)$ to be the convex
body $C_t\sm (g(v)\cup\{0\})$.
The proof of the coconvex Aleksandrov--Fenchel inequality is based
on the comparison between the linear families $\alpha$ and $\beta$.

We have the following relation between the polynomials $\Vol_\alpha$ and $\Vol_\beta$:
$$
\Vol_\alpha(v,t)=\Vol(C_t)-\Vol_\beta(v),\eqno{(V)}
$$
which is clear from the additivity of the volume.
The first term in the right-hand side has the form $ct^d$, where $c$
is some positive constant.
The second term in the right-hand side does not depend on $t$.
It follows from $(V)$ that $\Vol_\beta$ is a homogeneous degree $d$ polynomial.

Let us mark some points $(v_1,s_1)$, $\dots$, $(v_{d-2},s_{d-2})$ in
$\Omega\times (t_0,t_1)$.
Apply the differential operator $\frac 2{d!}L_{(v_1,s_1)}\dots L_{(v_{d-2},s_{d-2})}$
to both sides of $(V)$.
We obtain that
$$
Q_\alpha(v,t)=c't^2-Q^C_\beta(v),\eqno{(Q)}
$$
where $c'$ is some positive constant (equal to $c s_1\cdots s_{d-2}$).
In the right-hand side of $(Q)$, we have the difference of two quadratic
forms, moreover, these two forms depend on disjoint sets of variables.

If $q_1$, $q_2$ are quadratic forms depending on disjoint sets of variables,
and $(k_1,\ell_1)$, $(k_2,\ell_2)$, respectively, are signatures of these forms,
then $q_1+q_2$ is a quadratic form of signature $(k_1+k_2,\ell_1+\ell_2)$.
We now apply this observation to identity $(Q)$.
The first term of the right-hand side, $c't^2$, has signature $(1,0)$.
The signature of the left-hand side is equal to $(1,\ell)$ for some $\ell\ge 0$,
by the classical Alexandrov--Fenchel inequality.
It follows that the signature of $Q^C_{\beta}$ is $(\ell,0)$, i.e. the
form $Q^C_{\beta}$ is non-negative.

\section{Proof of Theorem B and its corollaries}
In this section, we prove Theorem B and derive a number of corollaries from it.

\subsection{Proof of Theorem B}
Consider a closed strictly convex cone $C\subset\R^d$ with the apex at $0$
and a nonempty interior.
A closed convex subset $\Delta\subset C$ is said to be \emph{$C$-convex} if
$\Delta+C=\Delta$.
The following lemma gives the most important example of $C$-convex sets.

\begin{lem}
\label{l:Cconv}
  If $\Delta\subset C$ is a convex set, and the complement of $\Delta$ in
  $C$ is bounded, then $\Delta$ is $C$-convex.
\end{lem}

\begin{proof}
Since $0\in C$, we have $\Delta+C\supset\Delta$.
It remains to prove that $\Delta+C\subset\Delta$.
Assume the contrary: there are points $x\in\Delta$ and $y\in C$ such that
$x+y\not\in\Delta$.
Consider the line $L$ passing through the points $x$ and $x+y$.
Since $C\setminus\Delta$ is bounded, there are points of $\Delta$ in $L$
far enough in the direction from $x$ to $x+y$.
Thus $x+y$ separates two points of $L\cap\Delta$ in $L$.
A contradiction with the convexity of $\Delta$.
\end{proof}

Let $\xi$ be a linear functional on $\R^d$ such that $\xi\ge 0$ on $C$
and $\xi^{-1}(0)\cap C=\{0\}$.
Recall that, for any subset $A\subset C$, we let $A_t$ denote the set of all points
$a\in A$ such that $\xi(a)\le t$.

\begin{lem}
\label{l:sum}
  Let $\Delta\subset C$ be a closed subset such that $C\setminus\Delta$ is bounded.
  For all sufficiently large $t>0$ and all $s>0$, we have
  $\Delta_t+C_s=\Delta_{t+s}$.
\end{lem}

\begin{proof}
  If $x\in\Delta_t$ and $y\in C_s$, then $x+y\in\Delta$ by $C$-convexity
  (Lemma \ref{l:Cconv})
  and $\xi(x+y)\le t+s$ since $\xi(x)\le t$ and $\xi(y)\le s$.
  On the other hand, take $z\in\Delta_{t+s}$ and consider two cases.

\emph{Case 1}: we have $\xi(z)\le t$. Then $z\in\Delta_t$, and, setting $x=z$, $y=0$,
we obtain that $z=x+y$, $x\in\Delta_t$ and $y\in C_s$.

\emph{Case 2:} we have $\xi(z)>t$. If $\lambda=t/\xi(z)$, then $\xi(\lambda z)=t$.
We now set $x=\lambda z$, $y=(1-\lambda)z$.
The number $t$ is sufficiently large; thus we may assume that
all values of $\xi$ on $C\setminus\Delta$ are less than $t$.
Then $x\in\Delta_t$ and $y\in C_s$.
\end{proof}

\begin{prop}
\label{p:chi}
  For all sufficiently large $t$, we have
  $$
  \1_{\Delta_t}*\1_{C_t}^{-1}=\1_{\Delta_t}-\1_{C_t}+\1_{\{0\}}.
  $$
\end{prop}

\begin{proof}
It suffices to compute
$\1_{C_t}*(\1_{\Delta_t}-\1_{C_t}+\1_{\{0\}})$ and verify that it is equal to $\1_{\Delta_t}$.
Opening the parentheses in the former expression and using Lemma \ref{l:sum}, we obtain
$$
\1_{C_t+\Delta_t}-\1_{C_t+C_t}+\1_{C_t}=\1_{\Delta_{2t}}-\1_{C_{2t}}+\1_{C_t}.
$$
Since $t$ is sufficiently large, we have
$\1_{\Delta_{2t}}-\1_{C_{2t}}=\1_{\Delta_t}-\1_{C_t}$
(both sides are equal to the minus indicator function of $C\setminus\Delta$),
and we obtain the desired.
\end{proof}

Let $\Delta$ be a closed convex subset of $C$ such that $C\sm\Delta$ is bounded.
Consider the corresponding $C$-coconvex body $A=C\sm (\Delta\cup\{0\})$.
Then, for all sufficiently large $t$, we have
$$
-\1_{A}=\1_{\Delta_t}-\1_{C_t}+\1_{\{0\}},
$$
in particular, the right-hand side is independent of $t$, provided that $t$ is large enough.
It follows from Proposition \ref{p:chi} that $-\1_A$ is a virtual convex polytope
$\1_{\Delta_t}*\1_{C_t}^{-1}$.
Moreover, the same proposition implies that
$-\1_{(A\oplus B)}=(-\1_{A})*(-\1_{B})$.
Thus $A\mapsto -\1_{A}$ is a homomorphism from the semi-group of $C$-coconvex bodies
to the multiplicative group of virtual polytopes.
This completes the proof of Theorem B.

\subsection{Proof of Corollaries \ref{c:poly} and \ref{c:Ehrhart}}
Corollary \ref{c:poly} follows directly from Theorem B and the theorem of \cite{KPa}
that a polynomial valuation on integer convex polytopes defines a
polynomial function on the group of virtual integer polytopes.
Corollary \ref{c:Ehrhart} follows from Theorem B and from
Proposition \ref{p:inv} stated below.

For a $C$-convex polytope $A$ as above, we let $A^\bullet$ denote
the intersection of the closure of $A$ with the interior of $C$.

\begin{prop}
  \label{p:inv}
  The inverse of the virtual polytope $-\1_{A}$ in the multiplicative
  group of virtual polytopes is equal to
  $$
  (-1)^{d-1}\1_{A^\bullet}\circ\sigma+\1_{\{0\}}.
  $$
  Here $\sigma$ is the antipodal map taking $v$ to $-v$.
\end{prop}

\begin{proof}
   For a subset $X\subset\R^d$, we let $\Int(X)$ denote the interior of $X$.
  According to \cite{KPa}, the inverse of a virtual convex polytope $\varphi$
  is equal to $\star(\varphi)\circ\sigma$, where the additive group homomorphism
  $\star:Z_c(\Z^d)\to Z_c(\Z^d)$ is defined uniquely by the property that
  $\star(\1_B)=(-1)^d\1_{\Int(B)}$ for every integer convex polytope $B$
  with nonempty interior.
  Setting $\varphi=-\1_{A}$, we obtain that
  $$
  \star(\varphi)=\star(\1_{\Delta_t}-\1_{C_t}+\1_{\{0\}})=
  (-1)^{d}(\1_{\Int(\Delta^t)}-\1_{\Int(C_t)})+\1_{\{0\}}=
  (-1)^{d-1}\1_{A^\bullet}+\1_{\{0\}}.
  $$
  The desired claim follows.
\end{proof}

\subsection{Proof of Theorem C}
Let $C$, $A$, $\Delta$ and $\xi$ be as above.
We assume that $C$ is an integer cone and that $\Delta$ has integer vertices.
Let $t>0$ be sufficiently large, so that $A$ is contained in $C_{t'}$ for some $t'<t$.

Clearly, there is a homomorphism $\Phi$ from the additive group $Z(\Z^d)$
to the additive group of rational functions of $x_1$, $\dots$, $x_d$ such that
$\Phi(\1_Q)=G(Q)$ for every integer cone $Q$.
Using the expression
$\1_{A}=-\1_{\Delta_t}+\1_{C_t}-\1_{\{0\}}$,
we obtain that
$$
G(A)=-G(\Delta_t)+G(C_t)-1.
$$
Apply Brion's theorem to $G(\Delta_t)$ and $G(C_t)$.
Near all vertices of $C_t$, except $0$, the polytopes $\Delta_t$ and $C_t$ coincide.
Therefore, the terms associated with these vertices cancel each other.
What remains are terms associated with vertices of $\Delta_t$ that
are simultaneously vertices of $A$ and the term with the vertex $0$ of $C_t$
(which is equal to $G(C)$).
Thus we have
$$
G(A)=-\sum_{a\in\vrt(A)\sm\{0\}} x^a G(C'_a)+G(C)-1,
$$
as desired.

\subsection{A viewpoint on Theorem A through virtual polytopes}
We will now interpret Theorem A in more conceptual terms involving Theorem B.
The function $\Vol$ assigning to every convex polytope its volume extends to
a valuation on convex polytopes, i.e., an additive group homomorphism
$\mu:Z_c(\R^d)\to\R$ such that $\mu(\1_{\Delta})=\Vol(\Delta)$ for every convex
polytope $\Delta$.
By \cite{KPb}, the restriction of $\mu$ to the multiplicative group $Z^*_c(\R^d)$
of virtual convex polytopes is a homogeneous degree $d$ polynomial.
It follows that there is a unique symmetric map
$$
\mu:Z^*_c(\R^d)\times\cdots\times Z^*_c(\R^d)\to\R
$$
(slightly abusing the notation, we use the same letter $\mu$ to denote this map)
with the property that
$\mu(\alpha,\dots,\alpha)=\mu(\alpha)$
for every $\alpha\in Z_c^*(\R^d)$ and that
$$
\mu(\alpha_1*\beta,\alpha_2,\dots,\alpha_d)=\mu(\alpha_1,\alpha_2,\dots,\alpha_d)+
\mu(\beta,\alpha_2,\dots,\alpha_d)
$$
for every $\alpha_1$, $\dots$, $\alpha_d$, $\beta\in Z^*_c(\R^d)$.
The number $\mu(\alpha_1,\dots,\alpha_d)$ is called the \emph{mixed volume
of virtual convex polytopes $\alpha_1$, $\dots$, $\alpha_d$}.
This is justified by the observation that if $\alpha_i=\1_{\Delta^{(i)}}$
for some convex polytopes $\Delta^{(i)}$, then the mixed volume of
$\alpha_i$s is equal to the mixed volume of $\Delta^{(i)}$s.
It is well known that the function $\mu$ is continuous with respect
to a natural topology on $Z^*_c(\R^d)$, in which $\1_{\Delta^{(n)}}\to\1_{\Delta}$
if convex polytopes $\Delta^{(n)}$ converge to a convex polytope
$\Delta$ in the Hausdorff metric.

Let $C$ be a closed strictly convex cone with the apex at $0$ and a nonempty interior.
Take a linear functional $\xi$ such that $\xi\ge 0$ on $C$ and $\xi^{-1}(0)\cap C=\{0\}$.
Suppose that a convex subset $\Delta\subset C$ is such that $C\sm\Delta$ lies in
the half-plane $\xi< t_0$.
We have
$$
\pd{}t\mu(\1_{\Delta_t},\dots,\1_{\Delta_t})|_{t=t_0}=
\frac d{t_0T}\left(\mu(\1_{\Delta_{t_0+t_0T}},
\1_{\Delta_{t_0}},\dots,\1_{\Delta_{t_0}})-\mu(\1_{\Delta_{t_0}},\1_{\Delta_{t_0}}\dots,
\1_{\Delta_{t_0}})\right)
$$
for arbitrary $T$.
We now let $T$ diverge to infinity, and observe that the
convex polytope $\frac 1T\Delta_{t_0(T+1)}$ converges to $C_{t_0}$.
Thus we obtain that
$$
\pd{}t\mu(\1_{\Delta_t},\dots,\1_{\Delta_t})|_{t=t_0}=
\frac d{t_0}\mu(\1_{C_{t_0}},
\1_{\Delta_{t_0}},\dots,\1_{\Delta_{t_0}}).
$$
It follows that the left-hand side does not depend on the geometry of
$\Delta$, it only depends on $C$ and on $t_0$.
In particular, the $d$ copies of $\Delta$ in the left-hand side
can be replaced with $d$ different closed convex sets
$\Delta^{(1)}$, $\dots$, $\Delta^{(d)}$ such that
the complements $C\sm\Delta^{(i)}$ are contained in the set $\{\xi<t_0\}$:
\begin{equation}
\label{e:indep}
\mu(\1_{C_{t_0}},\1_{\Delta^{(2)}_{t_0}},\dots,\1_{\Delta^{(d)}_{t_0}})=
\mu(\1_{C_{t_0}},\1_{\Delta_{t_0}},\dots,\1_{\Delta_{t_0}})
\end{equation}
As the convex set $\Delta^{(2)}_{t_0}$ degenerates to the cone $C_{t_0}$, e.g. through the
family $\frac 1T\Delta^{(2)}_{Tt_0}$, where $T\to\infty$, we obtain
that
\begin{equation*}
\mu(\1_{C_{t_0}},\1_{\Delta^{(2)}_{t_0}},\1_{\Delta^{(3)}_{t_0}},\dots,\1_{\Delta^{(d)}_{t_0}})=
\mu(\1_{C_{t_0}},\1_{C_{t_0}},\1_{\Delta^{(3)}_{t_0}},\dots,\1_{\Delta^{(d)}_{t_0}})
\end{equation*}
or, subtracting the right-hand side from the left-hand side,
\begin{equation}
\label{e:CA1}
  \mu(\1_{C_{t_0}},\1_{\Delta^{(2)}_{t_0}}*\1_{C_{t_0}}^{-1},
  \1_{\Delta^{(3)}_{t_0}},\dots,\1_{\Delta^{(d)}_{t_0}})=0.
\end{equation}

Set $A^{(i)}=C\sm(\Delta^{(i)}\cup \{0\})$ to be the $C$-coconvex polytopes
corresponding to $\Delta^{(i)}$.
Recalling that $-\1_{A^{(i)}}=\1_{\Delta^{(i)}_{t_0}}*\1_{C_{t_0}}^{-1}$, we can rewrite
equation (\ref{e:CA1}) as
\begin{equation}
\label{e:CA2}
  \mu(\1_{C_{t_0}},-\1_{A^{(i)}},\dots)=0.
\end{equation}
Here dots replace any sequence of $d-2$ virtual convex polytopes from the
group generated by $\1_{\Delta^{(i)}}$.

Taking $\Delta$ as above and setting $A=C\sm(\Delta\cup\{0\})$, we obtain that
$$
\mu(-\1_A)=-\mu(\1_A)=-\mu(\1_{C_{t_0}})+\mu(\1_{\Delta_{t_0}})
$$
from the additivity of volume.
Passing to the mixed volumes, we obtain that
\begin{equation}
\label{e:mixed1}
  \mu(-\1_{A^{(1)}},\dots,-\1_{A^{(d)}})=-\mu(\1_{C_{t_0}})+
  \mu(\1_{\Delta^{(1)}_{t_0}},\dots,\1_{\Delta^{(d)}_{t_0}}).
\end{equation}
Setting $\Delta=C$ in (\ref{e:indep}) and substituting into (\ref{e:mixed1}),
we obtain that
\begin{equation}
  \mu(-\1_{A^{(1)}},-\1_{A^{(2)}},\dots,-\1_{A^{(d)}})=
  \mu(-\1_{A^{(1)}},\1_{\Delta^{(2)}_{t_0}}\dots,\1_{\Delta^{(d)}_{t_0}}).
\end{equation}
From the right-hand side, we can subtract
$\mu(-\1_{A^{(1)}},\1_{C_{t_0}},\1_{\Delta^{(3)}_{t_0}},\dots,\1_{\Delta^{(d)}_{t_0}})$
(which is equal to zero by (\ref{e:CA2})) to obtain that
\begin{equation}
\label{e:AFform}
  \mu(-\1_{A^{(1)}},-\1_{A^{(2)}},-\1_{A^{(3)}},\dots,-\1_{A^{(d)}})=
\mu(-\1_{A^{(1)}},-\1_{A^{(2)}},\1_{\Delta^{(3)}_{t_0}},\dots,\1_{\Delta^{(d)}_{t_0}}).
\end{equation}
Fix the sets $\Delta^{(i)}$ with $i\ge 3$.
Then the left-hand side of (\ref{e:AFform}) can be viewed (up to a sign)
as the coconvex Aleksandrov--Fenchel from.
In fact, it is equal to the minus mixed volume of $C$-coconvex bodies $A^{(i)}$.
On the other hand, the right-hand side is a usual convex Aleksandrov--Fenchel
form associated with convex polytopes $\Delta^{(i)}_{t_0}$, $i\ge 3$, evaluated
at virtual polytopes $-\1_{A^{(1)}}$ and $-\1_{A^{(2)}}$.
Let $B$ denote this convex Aleksandrov--Fenchel form.
It follows from the convex Aleksandrov--Fenchel inequality that $B(\alpha,\alpha)\le 0$
provided that $\alpha\in Z_c^*(\R^d)$ is orthogonal to some convex polytope $\Delta$
in the sense that $B(\alpha,\1_{\Delta})=0$.
But $-\1_{A^{(1)}}$ is orthogonal to $C_{t_0}$ by (\ref{e:CA2})!
It follows that $B(-\1_{A^{(1)}},-\1_{A^{(1)}})\le 0$.
This implies the coconvex Aleksandrov--Fenchel inequality.

Thus we obtained another proof of Theorem A.
Although this proof is no simpler than the one given in Section \ref{ss:AF},
it reveals the role of virtual convex polytopes and the fact that
the coconvex Aleksandrov--Fenchel form is no different from the
convex Aleksandrov--Fenchel form evaluated at certain virtual polytopes.

\begin{rem}
  We now sketch an analogy, which can be easily formalized and which may shed some
  light to the argument presented above.
  Consider complex algebraic varieties $X$, $Y$ and a regular map $f:X\to Y$.
  Fix a point $y_0\in Y$, and assume that $f:X\sm f^{-1}(y_0)\to Y\sm\{y_0\}$
  is an isomorphism. 
  A Cartier divisor $D$ in $X$ is said to be \emph{sub-exceptional} if
  the support of $D$ maps to $y_0$ under $f$.
  Sub-exceptional divisors in $X$ correspond to coconvex polytopes
  (under the analogy, which we are discussing).
  A Cartier divisor $D$ in $X$ is said to be \emph{off-exceptional}
  if $D=f^*(\widetilde D)$ for some Cartier divisor $\widetilde D$ in $Y$,
  whose support does not contain $y_0$.
  Off-exceptional divisors correspond to the cone $C$ (we are looking
  at the local geometry of $Y$ near $y_0$, thus we do not distinguish between
  different off-exceptional divisors).
  Let $S$ and $O$ be a sub-exceptional and an off-exceptional divisors, respectively.
  If $[S]$ and $[O]$ stand for the classes of these divisors in the Chow ring,
  then obviously $[S]\cdot [O]=0$, since the supports of $S$ and $O$ are disjoint. 
  This fact is analogous to equation (\ref{e:CA2}).
  As is shown in \cite{Kh}, the Hodge index theorem implies the 
  Aleksandrov--Fenchel inequalities for the intersection of divisors in $X$
  (under some natural assumptions on $X$, e.g., when $X$ is projective and smooth).
  These inequalities can be used to provide an analog of Theorem A for the 
  intersection form on sub-exceptional divisors. 
  To make the described analogy into a precise correspondence, one 
  takes $X$ and $Y$ to be toric varieties associated with integer polytopes
  $\Delta_{t_0}$ and $C_{t_0}$. 
\end{rem}


\begin{thebibliography}{99999}
\bibitem[Al]{Al}
A.D. Aleksandrov, {\em To the theory of mixed volumes of convex bodies II.
New inequalities between mixed volumes and their applications''} (Russian),
Matem. Sb., \textbf{2} (1937), 6, pp. 1205--1238

\bibitem[AVG]{AVG}
V. Arnold, A. Varchenko, S. Gusein--Zade,
\emph{Singularities of Differentiable Maps},
Springer (1985)

\bibitem[B]{B}
D. Bernstein, \emph{The number of roots of a system of equations},
Functional Analysis and Its Applications \textbf{9} (1975), 3, pp. 183--185

\bibitem[BS]{BS}
M. Beck, S. Robins, \emph{Computing the Continuous Discretely},
Springer (2009)

\bibitem[Br]{Br}
M. Brion, \emph{Points entiers dans les poly\`edres convexes}. Ann. Sci. \'Ecole
Norm. Sup. (4), \textbf{21}:4 (1988), pp. 653--663

\bibitem[Fi]{Fi}
F. Fillastre, \emph{Fuchsian convex bodies: basics of Brunn--Minkowski theory},
Geom. Funct. Anal., \textbf{23}:1 (2013), pp. 295--333

\bibitem[Kh]{Kh}
A. Khovanskii, \emph{Algebra and mixed volumes},
in: Yu.D. Burago, V.A. Zalgaller, \emph{Geometric inequalities}, Springer (2010)

\bibitem[Ko]{Ko}
A. Kouchnirenko, \emph{Newton polyhedron and Milnor numbers},
Functional Analysis and Its Applications \textbf{9} (1975), 1, pp. 71--72

\bibitem[KKa]{KKa}
K. Kaveh, A.G. Khovanski\u{\i},
\emph{Newton--Okounkov bodies, semigroups of integral points, 
graded algebras and intersection theory}, Ann. Math. \textbf{176} (2012), Issue 2, pp. 925--978.

\bibitem[KKb]{KK}
K. Kaveh, A.G. Khovanski\u{\i},
\emph{Convex bodies and multiplicities of ideals}, preprint (2013),
to appear in the Proceedings of the Steklov Institute of Mathematics.

\bibitem[KPa]{KPa}
A. Khovanski\u{\i}, A. Pukhlikov,
\emph{Finitely additive measures of virtual polytopes}.
St. Petersburg Math. J. \textbf{4}:2 (1993), pp. 337--356

\bibitem[KPb]{KPb}
A. Khovanski\u{\i}, A. Pukhlikov,
\emph{The RiemannÂRoch theorem for integrals and sums of quasipolynomials on virtual polytopes}.
St. Petersburg Math. J., \textbf{4}:4 (1993), pp. 789--812	

\bibitem[KT]{KT}
A. Khovanski\u{\i}, V. Timorin,
\emph{Aleksandrov--Fenchel inequality for coconvex bodies}, preprint
\verb!arXiv:1305.4484!

\bibitem[KV]{KV}
A. Khovansky, A. Varchenko,
\emph{Asymptotics of Integrals Over Vanishing Cycles and the Newton Polyhedron},
Soviet Math. Doklady \textbf{32} (1985), pp. 122--127.

\bibitem[Mc77]{Mc77}
P. McMullen,
\emph{Valuations and Euler-type relations on certain classes of convex polytopes},
Proc. London Math. Soc., Ser. 3, \textbf{35}, 1 (1977), pp. 113--135

\bibitem[Mc93]{Mc}
P. McMullen, {\em On simple polytopes}, Invent. math. \textbf{113} (1993),
pp. 419--444

\bibitem[T]{T}
B. Teissier, \emph{Du th\'eor\`eme de lindex de Hodge aux in\'egalit\'es isop\'erim\'etriques}.
C. R. Acad. Sci. Paris S\'er. A, \textbf{288} (1979), 4, 287--289.

\bibitem[VT]{VT}
V. Timorin,
{\em An analogue of the Hodge-Riemann relations for simple polytopes},
Russian Mathematical Surveys, \textbf{54}, No.2 (1999), pp. 381--426
\end{thebibliography}
\end{document}